\documentclass[11pt,a4paper]{amsart}

  \usepackage[english]{babel}
   \usepackage{amsfonts,amssymb,amsmath,amsthm}
    \usepackage{latexsym,mathrsfs}
     \usepackage{fixme}
     \usepackage{xcolor}

\input amssym.def
\numberwithin{equation}{section}

\newcommand{\R}{\mathbb R}

\newcommand{\eps}{\varepsilon}

\newcommand{\lto}{\left(}
\newcommand{\rto}{\right)}

                          \newcommand{\dx}{\mathrm{d}x}

  \newcommand{\Rn}{\mathbb{R}^n}
   \newcommand{\N}{\mathbb{N}}
    \newcommand{\Om}{\Omega}
     \newcommand{\test}{C_c^\infty}
      \newcommand{\prt}{\partial}
       \newcommand{\omn}{\omega_{n}}
        \newcommand{\dd}{\mathrm{d}}
         
          \newcommand{\lapd}{\Delta d}
           \newcommand{\lapdf}{\Delta_F d_F}

                  \newcommand{\Ln}{\mathcal{H}^n}
                   \newcommand{\Hnmo}{\mathcal{H}^{n-1}}
                    \newcommand{\intom}{\int_{\Om}}

\newcommand{\be}{\begin{equation}}
 \newcommand{\ee}{\end{equation}}
  \newcommand{\bea}{\begin{eqnarray}}
   \newcommand{\eea}{\end{eqnarray}}
    \newcommand{\bean}{\begin{eqnarray*}}
     \newcommand{\eean}{\end{eqnarray*}}
\DeclareMathOperator{\dist}{dist}
 
  \DeclareMathOperator{\Div}{div}

\newtheorem{theorem}{Theorem}[section]
 \newtheorem{proposition}[theorem]{Proposition}
  \newtheorem{lemma}[theorem]{Lemma}
   \newtheorem{corollary}[theorem]{Corollary}
    \theoremstyle{definition}
     \newtheorem{definition}[theorem]{Definition}
      
       \newtheorem{remark}[theorem]{Remark}
        
\title[Weighted anisotropic Sobolev type inequality]{A weighted anisotropic Sobolev type inequality\\ and its applications to Hardy inequalities}
        \date{\today}
\author[di Blasio]{Giuseppina di Blasio}
\author[Pisante]{Giovanni Pisante}
\author[Psaradakis]{Georgios Psaradakis}

\address[G. di Blasio \& G. Pisante]{Dipartimento di Matematica e Fisica
	\newline\indent
	Universit\`a degli Studi della Campania "L. Vanvitelli"
	\newline\indent
	Viale Lincoln 5, 81100 Caserta, Italy}
\email{giuseppina.diblasio@unicampania.it}
\email{giovanni.pisante@unicampania.it}

\address[G. Psaradakis]{Lehrstuhl f\"{u}r Mathematik IV
	\newline\indent
	WIM -  Universit\"{a}t Mannheim
	\newline\indent
	B6 28 - 68159 Mannheim, Germany}
\email{psaradakis@uni-mannheim.de}

\begin{document}
\maketitle
%
%
%
%
\begin{abstract}
In this paper we focus our attention on an embedding result for a weighted Sobolev space that involves as weight the distance function from the boundary taken with respect to a general smooth gauge function $F$. Starting from this type of inequalities we prove some refined Hardy-type inequalities.
\end{abstract}

\begin{center}
	\begin{minipage}{10cm}
		\tableofcontents
	\end{minipage}
\end{center}
%

\section{Introduction}\label{intro}
Let $\Om$ be an open, connected, proper subset of $\Rn$ where $n\geq 2$. In this paper we focus our attention on an embedding result for a weighted Sobolev space that  involves as weight the distance function from the boundary, $\partial \Omega$, taken with respect to a general smooth gauge function  $F$.  More precisely we prove that there exists a constant $C>0$ such that for any $v\in\test(\Om)$
\begin{equation}\label{eq:general-sobolev-intro}
\|d_F^{1/p'-\alpha}v\|^p_{L^{p^{*,\alpha}}(\Om)}\leq C \intom\left(d_{F}^{-1}|v|^p+d_{F}^{p-1}F^p(\nabla v)\right)\dx,
\end{equation}
where  $d_F$ denotes the anisotropic distance function to $\prt\Om$,  $0\leq \alpha <1$, $1\leq p<\frac{n}{1-\alpha}$, $p^{*,\alpha}:=\frac{np}{n-p(1-\alpha)}$ and $C=C(n,p,\alpha,F)>0$ (see Theorem \ref{thm-main}). When $\alpha=0$ this result is closely related to the classical Sobolev inequality being $p^{*,0}=p^*$ (see Remark \ref{rem:Sobolev-Classic}). We point out that \eqref{eq:general-sobolev-intro} can be read as a continuous embedding between weighted spaces (see \cite[Section 8.10]{KJF} and Remark \ref{rem:Kufner}) since it amounts writing
\[
W_0^{1,p}(\Omega,d_F^{\omega_\iota}) \hookrightarrow L^q\Big(\Omega;d_F^{(1/p'-\alpha)p^{*,\alpha}}\Big),\hspace{1em}\omega_\iota=(-1,p-1).
\]

Our interest in this type of inequalities has been mostly motivated by their deep connections with refined Hardy-type inequalities. This link has been extensively investigated in the Euclidean case for instance in \cite{FMT,FPs,FrL}. In fact, a careful look at their proofs shows that the crucial inequalities actually rely on an application of suitable versions of \eqref{eq:general-sobolev-intro}.
More precisely it has been proven  in \cite{FMT} (for the case $\alpha=0$) and lately in \cite{FPs} (for $0\leq \alpha< 1$) that, if $\Om$ is a uniformly $C^2$ mean convex domain having finite volume, then for any $v\in\test(\Om)$ it holds
\begin{equation}\label{ineq.FPs}
C\big\|d^{1/p'-\alpha}v\big\|^p_{L^{p^{*,\alpha}}(\Om)} \leq  \intom d^{p-1}|\nabla v|^p~\dd x
+
\intom|v|^p(-\lapd)\dd x,
\end{equation}
with a positive constant $C=C(n,p,\alpha,\Om)$. Let us point out that the mean convexity assumption on $\Omega$ ensures that $-\lapd$ can be considered as a nonnegative Radon measure and this seems to be a crucial point in obtaining estimate \eqref{ineq.FPs}. Analogously, we can prove that if $-\lapdf \geq 0$, where $\lapdf$ is the distributional anisotropic $F$-Laplacian of $d_F$ (see \eqref{F-lap}),  then the following anisotropic version of \eqref{ineq.FPs} is true for any $v\in\test(\Om)$:
\begin{equation}\label{eq:laplacian:intro}
C\|d_F^{1/p'-\alpha}v\|^{p}_{L^{p^{*,\alpha}}(\Omega)}
\leq\intom d_F^{p-1}F(\nabla v)^p~\dd x
+
\intom |v|^p (-\lapdf)\dd x,
\end{equation}
for some positive constant $C=C(n,p,\alpha,F,\Om)$ (cf. Corollary \ref{cor_02} for the proof under less restrictive assumptions on $\Omega$).

It is worth to recall that in the very special case of the half space, i.e. if $\Omega =\Rn_+:=\{x=(x_1,...,x_n)\in\Rn~|~x_n>0\}$, where we simply have  $d(x)=x_n$ and so $-\lapd=0$, choosing $\alpha=(p-1)(p+n-1)$, the corresponding inequality
\[
\int_{\Rn_+}x_n^{p-1}|\nabla v|^p~\dd x
\geq
C(n,p)\Big(\int_{\Rn_+}x_n^{p-1}|v|^{p(p+n-1)/(n-1)}\dd x\Big)^{(n-1)/(p+n-1)}
\]
had been established using different approaches for example in \cite[\S6]{MSh} for $p=2$, in \cite{CR-O,Ng} for general $p>1$ and in \cite{DLV} for the fractional case.

Improved versions of the anisotropic Hardy inequality for $p=2$ have been investigated in \cite{DdBG}. In the last section of our paper we adapt the (by now classical) method from \cite{BFT} to our case and, once \eqref{eq:laplacian:intro} has been established, we deduce a series of improved versions of the anisotropic Hardy inequality. More precisely, at first we prove that if $p\geq 2$, under the same assumptions that ensure the validity of \eqref{eq:laplacian:intro} (see Theorem \ref{Hardy-Sobolev-anis}), for any $u\in\test(\Om)$ the following Hardy-Sobolev estimate holds true for some positive constant $C=C(n,p,b,F,\Om)>0$
\[
C \left\|\frac{u}{d_F^{\alpha}}\right\|_{L^{p^{*,\alpha}}(\Om)}^{p} \leq \intom F(\nabla u)^p~\dd x-\Big(\frac{p-1}{p}\Big)^p\intom\frac{|u|^p}{d_F^p}\dd x.
\]
Then, assuming further that $\Omega$ is a uniformly Lipschitz domain with finite measure, we can prove a Hardy estimate of the anisotropic norm of the gradient in a suitable Morrey space (see Theorem \ref{thm-HM}). More precisely we prove there exists a constant $C=C(n,p,F,\Om)>0$ such that
\[
C \| F(\nabla v) \|_{L^{1, n-n/p}(\Om)} \leq \intom F(\nabla v)^p~\dd x-\Big(\frac{p-1}{p}\Big)^p\intom\frac{|v|^p}{d_F^p}\dd x.
\]
\section{Preliminaries}
In this section we will fix some notation and recall some well known facts about convex functions that will be useful for our later purpose. For a detailed treatment of this material we refer the reader to classical monographs as \cite{Roc70,Sch13}. Throughout the paper we assume that $F:\Rn\mapsto[0,\infty)$ is a \emph{strictly convex norm}, i.e. an even \emph{gauge} function positive except at the origin and such that $\{\xi\in \Rn \,:\; F(\xi)<1\}$ is a strictly convex set. In particular $F$ will be a non-negative positively homogeneous convex function with $F(0)=0$. Given the norm $F$, the \emph{polar} norm of $F$, $F^o$, is the closed gauge function defined by
\[
F^o(\eta):=\sup_{\xi\not=0}\frac{\xi\cdot \eta}{F(\xi)},~~~~~~\eta\in \Rn\setminus \{0\}.
\]
From the previous definition it immediately follows the useful inequality
\begin{equation}
\label{F-young}
|\langle \xi, \eta\rangle| \le F(\xi) F^{o}(\eta) \qquad \forall ~ \xi, \eta \in  \Rn.
\end{equation}
In particular we recall that if $F\in C^1(\R^n\setminus\{0\})$ the strict convexity of $F$ ensures $F^o$ satisfies the same regularity (see \cite[Corollary 1.7.3.]{Sch13} and also \cite{CS}). The following identities are easily proven (see for instance \cite[Lemma 3.1]{CS})
\begin{equation}
F\lto F_\xi^o(\xi)\rto=F^o(  F_\xi(\xi))=1\quad \forall ~ \xi \not=0. \label{eq:F1}
\end{equation}

To the polar norm $F^0$ we associate the  \emph{Minkowski metric} on $\R^n$ defined by
\[
d_{F^o}(x,y)= F^o(x-y).
\]
We recall that to any norm can be associated a  Minkowski metric on $\R^n$, i.e. a metric compatible with addition and scalar multiplication (indeed there is a one-to one correspondence, see \cite[Section 15]{Roc70}).

Throughout, a \emph{domain} will be a connected open set $\Omega \subset \R^n$, $n\geq2$, having non-empty boundary. A key role in our weighted inequalities will be played by the function that assigns to any point $x\in\bar\Omega$ its Minkowski distance from the boundary of $\Omega$, $\partial \Omega$. We will refer to this function as the \emph{anisotropic distance} and we will use the notation
\[
d_F(x):=\inf_{y\in\partial\Omega} d_{F^o}(x,y)=\inf_{y\in\partial\Omega} F^o(x-y), \quad x\in\bar\Om .
\]
It is clear that $d_F$ is a Lipschitz continuous function in $\bar{\Om}$. If $d_F\in L^\infty(\Om)$, therefore $d_F\in W_0^{1,\infty}(\Om)$, we say that $\Omega$ has finite \emph{$F$-anisotropic inner radius} defined by
\[
r_\Omega := \sup_{x\in\Om} d_F(x).
\]
Moreover $d_F$ satisfies the eikonal-type equation
\be\label{Fd}
F\big(\nabla d_F(x)\big)=1 \quad\text{a.e. in }\Omega.
\ee
For an extensive treatment of the properties of anisotropic distance functions we refer to \cite{CrM}.

We recall that the \emph{anisotropic $F$-Laplacian} of a function $u\in W^{1,1}_{loc}(\Om)$, denoted by $\Delta_Fu$, is defined as the distributional divergence of the $L^1_{loc}(\Om)$ vector field $F(\nabla u) F_\xi (\nabla u)$, i.e.
\begin{equation}\label{F-lap}
\Delta_Fu [\varphi] = -\int_\Omega  F(\nabla u)F_\xi(\nabla u) \cdot \nabla \varphi \, \dx\,,\hspace{1em}\forall~ \varphi \in\test(\Om).
\end{equation}
If $u\in C^2(\Om)$ than $\Delta_Fu$ is simply represented by the function
\[
\Delta_Fu(x)= \text{div}(F(\nabla u(x))F_\xi(\nabla u(x))).
\]

We will use the notation
\[
-\Delta_{F}u \ge 0 \quad\text{in }\mathcal{D}'(\Om),
\]
if
\[
\int_{\Omega}F(\nabla u)F_{\xi}(\nabla u) \cdot \nabla \varphi\, \dx \ge 0\hspace{1em}\forall ~ \varphi \in\test(\Omega),\,\varphi \ge 0.
\]

\section{Weighted anisotropic Sobolev inequality}
The main results of this section stem from the following functional version of the anisotropic isoperimetric inequality
\begin{equation}\label{Ga-Ni}
S_{n,F}\|f\|_{L^{\frac{n}{n-1}}(\Rn)}
\leq
\int_{\Rn}F(\nabla f)\dx
\hspace{1em}\forall~f\in\test(\Rn).
\end{equation}
The inequality \eqref{Ga-Ni} with a non sharp constant can be easily obtained from the classical Gagliardo-Nirenberg inequality by using the properties of the norm $F$. The more subtle result, with the best possible constant $S_{n,F}$, has been established in \cite{AlFTrL}. A local version of \eqref{Ga-Ni}, without the best constant, is also easily otained. Indeed, recalling that if $V\subset \R^n$ is any sufficiently smooth bounded domain then (see \cite[p. 189]{Mz})
\begin{equation*}
n\omn^{1/n} \|f\|_{L^{\frac{n}{n-1}}(V)} \leq \| \nabla f\|_{L^1(V)} + \|f\|_{L^1(\partial V)}\hspace{1em}\forall~f\in C^\infty(V),
\end{equation*}
where $\omn$ denote the volume of the unit ball in $\mathbb{R}^n$, again from the properties of $F$ it readily follows that there exists a positive constant $C_1=C_1(F)$ such that
\begin{equation}\label{Ga-Ni-Anis-bounded}
n\omn^{1/n} \|f\|_{L^{\frac{n}{n-1}}(V)} \leq  C_1 \int_{V}F(\nabla f)\dx + \|f\|_{L^1(\partial V)}.
\end{equation}

Given $p$, $\alpha$, $\beta\in \R$ with $\beta\not= p(1-\alpha)$, we define the number
\[
p_{\beta,\alpha}:=\frac{\beta p}{\beta-p(1-\alpha)}.
\]
Let us explicitly note that, if $\beta=n\in \N$, the definition of $p_{n,\alpha}$ extends the classical definition of critical Sobolev exponent. Indeed, for $1\leq p < n$ and $\alpha =0$ we simply have
\[
p_{n,0}=p^*:=\frac{np}{n-p}.
\]
When $\alpha$ runs in the interval $[0,1)$ the value of $p_{n,\alpha}$ runs in the interval $(p,p*]$. In general, for $\alpha \in [0,1)$ and $1\leq p < \frac{n}{1-\alpha}$ we clearly have $ p < p_{n,\alpha} < +\infty$ and we will will use the more friendly notation
\[
p^{*,\alpha}:=p_{n,\alpha}=\frac{n p}{n-p(1-\alpha)}.
\]

From now on we fix the dimension $n\geq 2$. We start with a technical lemma based on a suitable interpolation inequality and on \eqref{Ga-Ni}.

\begin{lemma}\label{lem:pre-sobolev}
Let $\Om\subset\Rn$ be a domain. For any $b\in \R$, $\alpha \in [0,1)$ and $\eps >0$, the following inequality holds
\be\label{LocalSobolev}
\begin{split}
\| d_F^{b} \, v \|_{L^{1^{*,\alpha}}(\Omega)}   \leq  & \; \frac{(1-\alpha)\eps^{-\frac{\alpha}{1-\alpha}}}{S_{n,F}}
 \int_{\Omega} d_F^{b+\alpha} F(\nabla v) \dx  \\   + &  \left(\frac{(1-\alpha)\eps^{-\frac{\alpha}{1-\alpha}}|b+\alpha|}{S_{n,F}}+\eps\alpha\right)  \int_{\Omega}  d_F^{b-(1-\alpha)} |v| \dx,
\end{split}
\ee
for any $v\in\test(\Omega)$, where $S_{n,F}$ is the constant in \eqref{Ga-Ni}.
\end{lemma}

\begin{proof} Given $v\in\test(\Omega)$, as first step, using H\"older's and Young's inequalities, we can prove the following interpolation estimate
\begin{equation}\label{interpolation-01}
\| d_F^{b} \, v \|_{L^{1^{*,\alpha}}(\Omega)} \leq (1-\alpha)\eps^{-\frac{\alpha}{1-\alpha}} \|d_F^{b+\alpha} v\|_{L^{1^*}(\Omega)}+\eps\alpha \|  d_F^{b-(1-\alpha)}  v \|_{L^{1}(\Omega)} .
\end{equation}
For that, we first observe that the numbers $\frac{1^*}{(1-\alpha)1^{*,\alpha}}$ and $\frac{1}{\alpha \,1^{*,\alpha}}$ are H\"older conjugate.
We can then apply H\"older's inequality with these exponents to the functions $d_F^{(b+\alpha)(1-\alpha)1^{*,\alpha}} |v|^{(1-\alpha) 1^{*,\alpha}}$ and $d_F^{[b-(b+\alpha)(1-\alpha)]1^{*,\alpha}} |v|^{\alpha\, 1^{*,\alpha}}$ and, since
 \[
\frac{b-(b+\alpha)(1-\alpha)}{\alpha} = b-(1-\alpha),
 \]
 we have
 \[
 \|d_F^{b} v \|_{L^{1^{*,\alpha}} (\Omega)}  \leq \|d_F^{b+\alpha} v\|^{1-\alpha}_{L^{1^{*}}(\Omega)} \| d_F^{b-(1-\alpha)}  v\|_{L_1(\Omega)}^{\alpha}.
\]
Now estimate \eqref{interpolation-01} readily follows from Young's inequality
\[
 X^{1-\alpha} Y^{\alpha} \leq (1-\alpha)\eps^{-\frac{\alpha}{1-\alpha}} X + \eps\alpha  Y.
\]

From the anisotropic Gagliardo-Nirenberg inequality \eqref{Ga-Ni} with $f=d_F^{b+\alpha} v$, using the Leibniz rule, the sub-linearity, the homogeneity of $F$ and \eqref{Fd}, we get
\be\label{Ga-Ni-Anis-d}
 S_{n,F}\|d_F^{b+\alpha}v\|_{L^{1^{*}}(\Omega)}\leq \int_{\Omega} d_F^{b+\alpha} F(\nabla v)\,\dx +|b+\alpha|\int_{\Omega}  d_F^{b-(1-\alpha)}|v|\dx .
\ee

The desired inequality \eqref{LocalSobolev} follows combining \eqref{Ga-Ni-Anis-d} and \eqref{interpolation-01}. 
  \end{proof}

\begin{remark}
	Under the same hypotheses of Lemma \ref{lem:pre-sobolev} it can be proved that for any sufficiently smooth bounded domain $V\subset \Omega$ it holds
	\[
	\begin{split}
		\| d_F^{b} \, v \|_{L^{1^{*,\alpha}}(V)}   \leq  & \; \frac{(1-\alpha)}{n\omn^{1/n}\eps^{\frac{\alpha}{1-\alpha}}} \left( C_1
		 \int_{V} d_F^{b+\alpha} F(\nabla v) \dx + \int_{\partial V} d_F^{b+\alpha} |v|\dd\Hnmo \right) \\   + &  \left(\frac{(1-\alpha)|b+\alpha|}{n\omn^{1/n}\eps^{\frac{\alpha}{1-\alpha}}}C_1+\eps\alpha\right)  \int_{V} d_F^{b-(1-\alpha)} |v| \dx,
	\end{split}
	\]
where $C_1$ is the constant from \eqref{Ga-Ni-Anis-bounded}. The proof follows using the same argument as in the proof of Lemma \ref{lem:pre-sobolev} but applying the Sobolev trace inequality \eqref{Ga-Ni-Anis-bounded} instead of \eqref{Ga-Ni}.
\end{remark}

Relying on Lemma \ref{lem:pre-sobolev} we can prove the following Sobolev type inequality.

\begin{theorem}\label{thm-main}
Let $\Om\subset\Rn$ be a domain. Let $\alpha \in [0,1)$ and  $p\in \big[ 1, \frac{n}{1-\alpha} \big)$. Then there exists a constant $C_2:=C_2(n,p,\alpha,F)>0$ such that for any $v\in\test(\Om)$ it holds
\begin{equation}\label{eq:general-sobolev}
	\left\|d_F^{1/p'-\alpha}v\right\|^p_{L^{p^{*,\alpha}}(\Om)} \leq C_2 \intom \left(\frac{|v|^p}{d_{F}}+d_{F}^{p-1} F^p(\nabla v)\right) \dx.
\end{equation}
\end{theorem}	

\begin{proof}
For $b>0$ and $s>1$ to be chosen later, we start by applying Lemma \ref{lem:pre-sobolev} replacing $v$ by $|v|^s$ in \eqref{LocalSobolev}, to obtain that there exists a constant $\kappa:=\kappa(n,b,\alpha,F)$ such that
\be \label{aux:01}
\kappa \, \left\| d_F^{b} \, v^{s} \right\|_{L^{1^{*,\alpha}}(\Omega)}   \leq   s  \int_{\Omega} d_F^{b+\alpha} |v|^{s-1} F(\nabla v) \dx  + \int_{\Omega}  d_F^{b-(1-\alpha)} |v|^{s} \dx.
\ee
 We next apply H\"{o}lder inequality in both terms of the right hand side to get
\be \label{aux:02}
\begin{split}
 \intom d_F^{b+\alpha} & |v|^{s-1}  F(\nabla v) \dx
   =\intom \left(d_F^{1/p'}F(\nabla v)\right)\left(d_F^{b+\alpha-1/p'}|v|^{s-1}\right)\dd x \\
   & = \left( \intom d_F^{p-1}F^{p}(\nabla v)\dd x\right)^{\frac{1}{p}} \left( \intom d_F^{(b+\alpha-1/p')p'}|v|^{(s-1)p'}\dd x\right)^{\frac{1}{p'}},
\end{split}
\ee
and
\be \label{aux:03}
\begin{split}
 \intom d_F^{b-(1-\alpha)} & |v|^{s}  \dx
   =\intom \left(d_F^{-1/p}|v|\right)\left(d_F^{b-(1-\alpha)+1/p}|v|^{s-1}\right)\dd x \\
   & = \left( \intom \frac{|v|^p}{d_{F}}\dd x\right)^{\frac{1}{p}} \left( \intom d_F^{(b+\alpha-1/p')p'}|v|^{(s-1)p'}\dd x\right)^{\frac{1}{p'}}.
\end{split}
\ee
Now we can choose $s$ and $b$ that solve respectively the equations $(s-1)p' = s \,1^{*,\alpha}$ and $(b+\alpha-1/p')p' = b\,1^{*,\alpha}$  , that is
\[
s=\frac{p^{*,\alpha}}{1^{*,\alpha}}
\;\;\;
\text{and}
\;\;\;
b=\frac{(p(1-\alpha)-1)(n-(1-\alpha))}{n-p(1-\alpha)} = \left( \frac{1}{p'}-\alpha\right)\frac{p^{*,\alpha}}{1^{*,\alpha}}.
\]
With this choice of $s$ and $b$, noting that
\[
 \left( \intom d_F^{(b+\alpha-1/p')p'}|v|^{(s-1)p'}\dd x\right)^{\frac{1}{p'}} =\left \| d_F^{b} \, v^{s} \right\|_{L^{1^{*,\alpha}}(\Omega)}^{\frac{1^{*,\alpha}}{p'}},
\]
we can combine \eqref{aux:01}, \eqref{aux:02} and \eqref{aux:03} to get
\[
\left\| d_F^{b} \, v^{s} \right\|_{L^{1^{*,\alpha}}(\Omega)}   \leq  \frac{s}{\kappa} 2^{\frac{p-1}{p}} \left( \intom \left( \frac{|v|^p}{d_{F}}+d_{F}^{p-1} F^p(\nabla v)\right) dx \right)^{\frac{1}{p}} \left\| d_F^{b} \, v^{s} \right\|^{\frac{1^{*,\alpha}}{p'}}_{L^{1^{*,\alpha}}(\Omega)}.
\]
We conclude the proof by observing that the previous inequality is equivalent to \eqref{eq:general-sobolev} with $C_2=  \left(\frac{s}{\kappa}\right)^{p} 2^{p-1}$ since
\[
1-\frac{1^{*,\alpha}}{p'} = \frac{n-p(1-\alpha)}{np-p(1-\alpha)}= \frac{1^{*,\alpha}}{p^{*,\alpha}}=\frac{1}{s}
\]
and
\[
\left\|d_F^{1/p'-\alpha}v\right\|_{L^{p^{*,\alpha}}(\Om)} = \left( \intom d_F^{(1/p'-\alpha)p^{*,\alpha}}v^{p^{*,\alpha}} \right)^{\frac{1}{p^{*,\alpha}}} = \left\| d_F^{b} \, v^{s} \right\|^{\frac{1}{s}}_{L^{1^{*,\alpha}}(\Omega)}.
\]
 \end{proof}

\begin{remark}\label{rem:Sobolev-Classic}
In the special case $\alpha=0$, the connection of the previous result with the classical Sobolev inequality is easily understood by applying it to the function $d_F^{1/p'}v$. Indeed, for $v\in C^\infty_c(\Omega)$, denoting by $S_{n,p}$ the Sobolev constant, we can write
\[
\begin{split}
 S_{n,p}\left\|d_F^{1/p'}v\right\|_{L^{p^*}(\Omega)}  &  \leq \left\| \nabla(d_F^{1/p'}v) \right\|_{L^{p}(\Omega)} \\
  & \hspace{-1cm} \leq  \frac{1}{p'}\left\|d^{-1/p}_F\nabla d_F\, v\right\|_{L^{p}(\Omega)}  +	\left\| d_F^{1/p'} \nabla v\right\|_{L^{p}(\Omega)} \\
  & \hspace{-1cm} = \frac{1}{p'}\left(\intom d_{F}^{-1}|v|^p |\nabla d_F|^p \dx\right) ^{1/p} + \left(  \intom d_{F}^{p-1} |\nabla v|^p \dx\right)^{1/p}.
\end{split}
\]
Therefore, raising to power $p$ and using the properties of the norm $F$, we get
\[
 \left\|d_F^{1/p'}v\right\|^p_{L^{p^*}(\Omega)} \leq C(n,p,F)  \intom \left(d_{F}^{-1}|v|^p+d_{F}^{p-1} F^p(\nabla v)\right) \dx \hspace{1em} \forall v\in\test(\Om).
\]
This is \eqref{eq:general-sobolev} in the special case $\alpha=0$.
\end{remark}

\begin{remark}\label{rem:Kufner}
In the Euclidean case $F(\cdot)=|\cdot|$ and therefore considering $d$ to be the Euclidean distance, embedding theorems for weighted Sobolev spaces has been widely considered in the literature. In particular we can refer to \cite[Section 8.10]{KJF} for a result closely related to our Theorem \ref{thm-main}. To make this connection more precise we need to introduce some notation on weighted Sobolev spaces. For the sake of simplicity we restrict our treatment to spaces that involve only first order derivatives and we consider a special class of weighted Sobolev spaces. Let $\Omega$ be a domain in $\mathbb{R}^n$, $M\subset \bar\Omega$ with zero Lebesgue measure and let $d(x,M)=\dist(x,M)$. Given, for $\iota$ a multi-index of order 2, $\lambda_\iota=(\lambda_0,\lambda_1) \in \mathbb{R}^2$, we denote with
	\[
	W^{1,p}(\Omega,M,t^\lambda)
	\]
	the Banach space of all functions $u$ defined a.e. in $\Omega$ whose distributional derivatives $\nabla_\iota u$ belong to the weighted Lebesgue space $L_p\big(\Omega, d(x,M)^{\lambda_\iota}\big)$, i.e.
	\[
	\int_{\Omega} |\nabla_\iota u(x)|^p d(x,M)^{\lambda_\iota} \, \dd x < \infty.
	\]
	As in the classical case we define
	\[
	W_0^{1,p}(\Omega,M,t^\lambda) := \overline{\test(\Om)}^{W^{1,p}(\Omega,M,t^\lambda)}.
	\]
	The most common cases in the literature are $M=\partial \Omega$ and $M={x_o}$ with $x_o \in \bar \Omega$. In the first case we can use the shorter notation $W^{1,p}(\Omega,\partial \Omega,t^\lambda)=W^{1,p}(\Omega,d^\lambda)$.
	When $\lambda_0=\lambda_1=s > p-1$, $0<p<\infty$ and $\Omega$ is of class $C^{0,1}$, it can be proved that
	\[
	W^{1,p}(\Omega,d^{\omega_\iota})=W^{1,p}_0(\Omega,d^s) \hookrightarrow L^p(\Omega,d^{s-p}).
	\]
	where $\omega_\iota=(\omega_0,\omega_1)=(s-p,s)$ (see for example \cite[Theorems 8.10.12 and 8.10.14]{KJF}). The limit case $s=p-1$ is more delicate and the characterization of $W_0^{1,p}(\Omega,M,t^\lambda)$ in terms of weighted spaces with distance function is not as clean and it needs a logarithmic correction (see \cite[Remark 8.10.13]{KJF}).
	
	Nevertheless, in a similar spirit, \eqref{eq:general-sobolev} can be seen as a continuous embedding in the limit case $s=p-1$ (and consequently $\omega_\iota=(-1,p-1)$), of the anisotropic weighted Sobolev space $W_0^{1,p}(\Omega,d_F^{\omega_\iota})$ in a Lebesgue space weighted with the anisotropic distance $d_F$. More precisely we can rewrite the result as
	\[
	W_0^{1,p}(\Omega,d_F^{\omega_\iota}) \hookrightarrow L^{p^{*,\alpha}}\Big(\Omega,d_F^{(1/p'-\alpha)p^{*,\alpha}}\Big),
	\]
	and we get the result that can be traced back by carefully analyzing the proof of \cite[Section 4]{FPs}.
\end{remark}

Under suitable assumptions on $\Omega$, we show next how the previous result can be improved. To this end we introduce the following definition.

\begin{definition}\label{def:C-feeble-regularity}
Let $\Om \in \Rn$ be a domain, $\omega\Subset\Om$ and $K>0$. We say that \emph{$\Omega$ is $(\omega,K)$-feeble-regular} if for any $\varphi \in W^{1,\infty}_{c}({\Om\setminus \omega})$ with $\varphi \geq 0$ it holds
\begin{equation}
\label{bound-02}
\left| \int_{\Om\setminus \omega} \frac{1}{2}\nabla F^{2}(\nabla d_{F}) \cdot \nabla (\varphi d_{F}) \dd x \right| \leq K  \|\varphi\|_{L^1({\Om\setminus \omega})}.
\end{equation}
\end{definition}
\begin{remark}
The previous definition can be read as a weak-form of the assumption that $|d_F \Delta_F d_F|$ is bounded by the constant $K$ in the set $\Om \setminus \omega$.
\end{remark}

\begin{proposition} \label{cor_01}
Let $\Om\subset\Rn$ be a domain. Let $\alpha \in [0,1)$, $p\in \big[ 1, \frac{n}{1-\alpha} \big)$ and let
\be \label{def-a}
 a:= \frac{(p-1)(1-\alpha)(n-1)}{n-p(1-\alpha)}.
\ee
Suppose there exists $\delta >0$ and $\omega\Subset\Om$ such that $\Om$ is $(\omega,a-\delta)$-feeble-regular. Then for $0< \epsilon < \dist(\omega,\partial \Omega)$, there exists a constant $C_3=C_3(n,p,\alpha,F,\omega,\epsilon)>0$ such that
\[
 C_3 \|d_F^{1/p'-\alpha}v\|^p_{L^{p^{*,\alpha}}(\Om)}
  \leq
    \intom d_{F}^{p-1}F^p(\nabla v) \dx+\frac{1}{\epsilon}\int_{\omega_\epsilon\setminus\omega}d_F^{p-1}|v|^p \dx,
\]
for all $v\in\test(\Om)$, where $\omega_\epsilon:=\omega + B(0,\epsilon)$.
\end{proposition}

\begin{proof}
Consider a cut-off function $\phi : \Omega \to [0,1]$ such that $\phi(x)= 1$ for any $x\in \Omega \setminus \overline{\omega_\epsilon}$, $\phi(x)=0$ for  any $x\in \omega$ and $| \nabla \phi |_\infty \leq \kappa_1/\epsilon$ for a constant $\kappa_1>0$. We define $b$ as in the proof of Theorem \ref{thm-main}; that is
\[
b=\frac{(p(1-\alpha)-1)(n-(1-\alpha))}{n-p(1-\alpha)},
\]
and observe that $a=b+\alpha$.
   For $v \in \test(\Om)$, writing $v= \phi v + (1-\phi) v$, we have
\begin{equation}\label{eq:000}
 \|d^{b}_Fv\|_{L^{1^{*,\alpha}}(\Omega)}  \leq  \|d^{b}_F\phi v\|_{L^{1^{*,\alpha}}(\Omega \setminus \omega)}+\|d^{b}_F (1-\phi)v\|_{L^{1^{*,\alpha}}(\omega_\epsilon)}.
\end{equation}

Let us start by considering the first term of the right hand side of \eqref{eq:000}.  Using Lemma \ref{lem:pre-sobolev} we can easily deduce the existence of a constant $\kappa_{2}=\kappa_{2}(n,p,b,F)$ such that
\begin{equation}\label{eq:001}
\kappa_{2} \|d^{b}_F\phi v\|_{L^{1^{*,\alpha}}(\Omega \setminus \omega)}
  \leq
    \int_{\Omega \setminus \omega}d_F^{a}F(\nabla (\phi v))\dx +  \int_{\Omega  \setminus \omega}  d_F^{a-1} |\phi v| \dx.
\end{equation}
Now we want to estimate the last term of the right hand side of the previous inequality. We use the fact that
	$F_{\xi}(\nabla d_F) \cdot \nabla d_F=F(\nabla d_F)=1$ and the assumptions on $\Omega$ to get
	\begin{equation*}
	\begin{split}
	\int_{\Omega \setminus \omega}  d_F^{a-1}  |\phi v| \dx  = &  \frac{1}{a}  \int_{\Omega \setminus \omega}  |\phi  v|  F_{\xi}(\nabla d_F) F(\nabla d_F) \cdot \nabla d^{a}_F  \dx \\
	= &  \frac{1}{a}  \int_{\Omega \setminus \omega}   F_{\xi}(\nabla d_F) F(\nabla d_F) \cdot \nabla ( d^{a}_F  | \phi v| )\dx \\
	& ~~~~ -\frac{1}{a}  \int_{\Omega \setminus \omega} d^{a}_F   F_{\xi}(\nabla d_F) F(\nabla d_F) \cdot   \nabla ( |\phi v| ) \dx  \\
	\leq &  \frac{a-\delta}{a}  \int_{\Omega \setminus \omega}  d^{a-1 }_F  |\phi  v| \dx +  \frac{1}{a}  \int_{\Omega \setminus \omega} d^{a}_F    F( \nabla  |\phi v| ) \dx
	\end{split}
	\end{equation*}
	Where in last term we have used \eqref{F-young} and \eqref{eq:F1} and the $(\omega,a-\delta)$-feeble-regulaliry of $\Om$ (by applying \eqref{bound-02} with $\varphi= d_F^{a -1} \phi v$). In view of the previous estimate we can write
	\[
	\int_{\Omega \setminus \omega}  d^{a-1 }_F  |\phi v| \dx  \leq \frac{1}{\delta}  \int_{\Omega \setminus \omega} d^{a}_F    F( \nabla  \phi v ) \dx
	\]
	that together with \eqref{eq:001} gives
\be
\label{eq:002}
\begin{split}
 \frac{\delta \kappa_{2}}{1+\delta}   \|d^{b}_F\phi v\|_{L^{1^{*,\alpha}}(\Omega \setminus \omega)}
   \leq  &  \int_{\Om \setminus \omega}d_F^{a}F(\nabla (\phi v))\dx =  \int_{\Om \setminus \omega_\epsilon}d_F^{a}F(\nabla v)\dx
  \\ &  +  \int_{\omega_\epsilon \setminus \omega}d_F^{a}F( \phi \nabla v + v \nabla \phi)\dx.
\end{split}
\ee

Now we take into account the second term of the right hand side of \eqref{eq:000}. Since $\overline{\omega_\epsilon}\subset \Omega$ is compact, we can use H\"older's inequality and \eqref{Ga-Ni}  to get
\begin{align}\nonumber
 \|d^{b}_F (1-\phi)v\|_{L^{1^{*,\alpha}}(\omega_\epsilon)}
  &\leq\|d_F ^{b 1^{*,\alpha}}\|^{1/1^{*,\alpha}}_{L^{\left(\frac{1^*}{1^{*,\alpha}}\right)'}(\omega_\epsilon)}\| (1-\phi)v\|_{L^{\frac{n}{n-1}}(\omega_\epsilon)}
   \\ \nonumber & \leq  \kappa_{3}  \int_{\omega_\epsilon}d_F^{a}F\big(\nabla ((1-\phi) v)\big)\dx,
\end{align}
with a constant $\kappa_{3}=\kappa_{3}\big(n,p,b,F,\dist(\omega,\partial\Om),\mathrm{diam}(\omega),\epsilon\big)$. Here we have also used the fact that in $\omega_\epsilon$, $d_F$ is bounded above, and also below away from $0$. Then, by the properties of $\phi$,
\be \label{eq:003}
\begin{split}
 \frac1\kappa_3 \! \|d^{b}_F (1-\phi)v\|_{L^{1^{*,\alpha}}\!(\omega_\epsilon)}
  & \!  \leq \! \int_{ \omega}\|d_F^{a}F(\nabla v)\dx\! + \! \int_{\omega_\epsilon \setminus \omega} \!d_F^{a}F\big( \nabla((1-\phi) v)\big)\dx
   \\  & \leq\!
    \int_{ \omega_\epsilon}\!d_F^{a}F(\nabla v)\dx \!+\!  \int_{\omega_\epsilon \setminus \omega}\!d_F^{a}F( \phi\nabla v- v \nabla \phi) \dx.
\end{split}
\ee
Let us estimate the last integral.  Again from the hypotheses on $\phi$ and by homogeneity and convexity properties of $F$ we get
	\begin{equation}\label{eq:004}
	\int_{\omega_\epsilon \setminus \omega}d_F^{a}F( \phi\nabla v- v \nabla \phi) \dx   \leq   \int_{\omega_\epsilon \setminus \omega}d_F^{a}F( \nabla v) \dx +  \int_{\omega_\epsilon \setminus \omega}d_F^{a}|v| F(\nabla \phi) \dx. 	
	\end{equation}
Using the uniform bound on the gradient of $\phi$, we have $|F(\nabla \phi)| \leq \kappa_{4}/\epsilon$ for a constant $\kappa_{4}>0$, and putting together \eqref{eq:000}, \eqref{eq:002}, \eqref{eq:003}, \eqref{eq:004}
\[
 \kappa_{5}\|d^{b}_F v\|_{L^{1^{*,\alpha}}(\Omega)}
  \leq
   \intom d_F^{a}F(\nabla v) \dx
	+\frac{1}{\epsilon}\int_{ \omega_\epsilon \setminus \omega }d_F^{a}|v| \dx,
\]
with the constant $\kappa_{5}=\kappa_{5}(n,p,b,F,\omega,\epsilon)>0$. Replacing in the previous inequalities $v$ by $|v|^s$ with
\[
 s= \frac{p^{*,\alpha}}{1^{*,\alpha}},
\]
we get
\[
 \kappa_5 \|d^{b}_F|v|^s\|_{L^{1^{*,\alpha}}(\Omega)}
  \leq
   s\intom d_F^{b+\alpha}|v|^{s-1}F(\nabla v) \dx+\frac{1}{\epsilon}\int_{ \omega_\epsilon \setminus \omega }d_F^{b+\alpha}|v|^s \dx.
\]
Applying H\"{o}lder's inequality in both terms of the right hand side, we conclude arguing exactly as in the end of the proof of Theorem \ref{thm-main}.
 \end{proof}

\begin{remark}\label{rem:reg}
	We point out that the analytical assumption of the previous result can be linked to regularity conditions on the domain $\Omega$. Indeed, it is easy to verify that if $\Omega$ is bounded and $\partial \Omega$ is sufficiently smooth (for example of class $C^2$) then $|d_F \Delta_F d_F|$ is bounded in a sufficiently small neighborhood of the boundary. Then of course the estimate in Proposition \ref{cor_01} holds true with a constant $C(n,p,\alpha,F,\Om)$.
\end{remark}

\begin{remark}\label{rem:half_plane}
	We remark that the result of Proposition \ref{cor_01} can be sharpened as soon as $\Omega$ is such that $\Delta_{F} d_F=0$, as for instance in the case of the half-space. Indeed in this case we can choose $\omega=\emptyset$ for any choice of $\delta$. This allows us to say that
	\[
	C(n,p,\alpha,F) \left\|d_F^{1/p'-\alpha}v\right\|^p_{L^{p^{*,\alpha}}(\Omega)}
	\leq
	\int_{\Omega} d_{F}^{p-1}F^p(\nabla v) \dx \hspace{1em} \forall~v\in\test(\Omega).
	\]
	The previous inequality generalizes to the anisotropic setting the results obtained in \cite{FPs} for the half-space in the isotropic case.
\end{remark}

Another useful consequence of Theorem \ref{thm-main} is the following corollary that extends \cite[Proposition 4.1]{FPs} to the anisotropic case.

\begin{corollary}\label{cor_02}
	Under the assumptions of Proposition \ref{cor_01}, suppose in addition that $ r_\Omega < \infty$ and
	\begin{equation}
	\label{H1}
	-\Delta_{F}d_{F} \ge 0 \quad\text{in }\mathcal{D}'(\Om).
	\end{equation}
Then there exists a constant $C_4=C_4(n,p,\alpha,F,\omega,\epsilon,r_\Omega)>0$ such that for any $v\in\test(\Om)$ we have
	\[
	C_4\|d_F^{1/p'-\alpha}v\|^{p}_{L^{p^{*,\alpha}}(\Omega)}
	\leq\intom d_F^{p-1}F(\nabla v)^p~\dd x
	+
	\intom F_\xi(\nabla d_F)\cdot\nabla (|v|^p)\dd x.
	\]
	\end{corollary}

\begin{proof}
	
	We use Proposition \ref{cor_01} and  we need to estimate the term
	\[
	\int_{\omega_\epsilon\setminus\omega}d_F^{p-1}|v|^p \dx.
	\]
	Recalling that $F_{\xi}(\nabla d_F) \cdot \nabla d_F=1$ a.e., integrating by parts, using \eqref{H1}, \eqref{F-young} and \eqref{eq:F1} we have
	\[
	\begin{split}
	\int_{\omega_\epsilon\setminus\omega}d_F^{p-1}|v|^p~dx &\leq \int_{\Omega}d_F^{p-1}|v|^p~\dd x =
	\frac{1}{p-1}\int_{\Omega}F_{\xi}(\nabla d_F) \cdot \nabla d_F^{p} |v|^p~\dx
	\\  & \hspace{-5em}  =
	-\frac{p}{p-1}\int_{\Omega}d_F^{p}|v|^{p-1}F_{\xi}(\nabla d_F) \cdot \nabla |v| \dx+\frac{1}{p-1}\int_{\Omega} d_F^{p}|v|^{p} (-\Delta_F d_F)\dx
	\\  & \hspace{-5em}  \leq
	\frac{p}{p-1}\int_{\Omega}d_F^{p}|v|^{p-1}F(\nabla |v|)+ \frac{r_\Omega^{p}}{p-1}\int_{\Omega} |v|^{p} (-\Delta_F d_F)\dx.
	\end{split}
	\]
	Now we consider the first term of the last inequality, using the Young's inequality and since $d_F<r_{\Omega}$, we get
	\begin{equation*}
	\begin{split}
	\int_{\Omega}d_F^{p}|v|^{p-1}F(\nabla |v|)\dx & \leq  r_{\Omega} \int_{\Omega}d_F^{p-1}|v|^{p-1}F(\nabla |v|)\dx
	\\
	&\leq
	r_{\Omega}\Big(\kappa(\epsilon)\int_{\Omega} F^p(\nabla |v|)d_F^{p-1}\dx +\epsilon \int_{\Omega} |v|^pd^{p-1}_F \dx\Big),
	\end{split}
	\end{equation*}
	Combining the previous estimates we have
	\begin{equation*}
	\kappa(p,r_{\Omega})\int_{\Omega}d_F^{p-1}|v|^p~\dd x \leq \int_{\Omega} F^p(\nabla |v|)d_F^{p-1} \dx+\int_{\Omega} |v|^{p} (-\Delta_F d_F)\dx
	\end{equation*}
	and the result follows.
 \end{proof}

\begin{remark}\label{rem:boundary_Geometry}
	There is a strict connection with assumption \eqref{H1} and the \emph{mean-convexity} property of $\Omega$. In the Euclidean case it is well known that, under suitable regularity condition on $\partial \Omega$ (which includes the one in Remark \ref{rem:reg}), the domain $\Omega$ is mean-convex (i.e. its mean curvature is non negative at any point of $\partial \Omega$) if and only if $-\Delta d \ge 0$ in $\mathcal{D}'(\Om)$ (cf. \cite{GP,Gr,LLL,Ps}).
	We stress that even if $F$ is equivalent to the Euclidean norm, condition $-\Delta d \ge0$ is not equivalent to the condition \eqref{H1} as proved in \cite{DdBG}. Questions on the link between \eqref{H1} and the geometric condition on $\Omega$ in the anisotropic setting are one of the topics that will be treated in a forthcoming paper of F. Della Pietra, G. di Blasio, N. Gavitone, G. Pisante and G. Psaradakis.
\end{remark}
\section{Some consequences}
In this section we apply the weighted anisotropic Sobolev inequality of \S3 to improve the anisotropic Hardy inequality. Several anisotropic versions of the Hardy inequality, with different singular weights, are known; see \cite{AFMTV,Ba,BF,DdBG,MST,Vn}. Here we focus on the case where the singular weight is a negative power of $d_F$ and prove a Hardy-Sobolev and a Hardy-Morrey inequality.
\subsection{Anisotropic Hardy-Sobolev inequality}
We start by proving, for the sake of completeness, the following lemma that quantifies the classical convexity inequality for $F^p$ (cfr. for instance \cite[Appendix]{Ln} for the Euclidean case).

\begin{lemma}
Let $F\in C^{2}(\Rn \setminus\{0\})$, even, positively $1$-homogenous and strongly convex, i.e. $\nabla^{2}F^{2}(\xi)$ is a positively defined matrix for any $\xi\not =0$.
Then for $p\ge 2$, there exists $\sigma_F\geq1$ such that
\begin{equation}\label{convex}
 F^p(\xi_1+\xi_2)\ge F^p(\xi_1)+\frac{2^{1-p}}{\sigma^p_F}F^p(\xi_2)+pF^{p-1}(\xi_1)F_{\xi}(\xi_1)\cdot \xi_2\hspace{1em}\forall~\xi_1,\xi_2\in\Rn.
\end{equation}
\end{lemma}
\begin{proof}

First we note that the space $\Rn$ endowed with the norm $F$, is $p$-uniformly convex in the sense of \cite{BCL}, i.e. there exists a constant $\sigma_{F}$ such that.
\begin{equation}\label{A1}
 F^p\Big(\frac{\xi_1+\xi_2}{2}\Big)+\frac{1}{\sigma_F^p}F^p\Big(\frac{\xi_1-\xi_2}{2}\Big)
  \leq
   \frac12F^p(\xi_1)+\frac12F^p(\xi_2)\hspace{1em}\forall~\xi_1,\xi_2\in\Rn.
\end{equation}
Indeed, following the argument of  \cite[Proposition 4.6]{O}, recalling that for any $w \in \Rn \setminus \{0\}$  and $v \in \Rn$ with $F(v)=1$ it holds
\[
\lim_{\eps \to 0} \frac{1}{2\eps^{2}} \left[  F^2 (w-\eps v) + F^2 (w+\eps v) - 2F^2(w)  \right] = \nabla^{2}F^{2}[w](v,v),
\]
we easily infer that
\[
\begin{split}
\inf_{v\in \Rn,\,F(v)=1}\, \inf_{w\in \Rn\setminus\{0\}} & \left\{  \lim_{\eps \to 0} \frac{1}{2\eps^{2}} \left[  F^2 (w-\eps v) + F^2 (w+\eps v) - 2F^2(w)  \right] \right\} \\
 & = \frac{1}{2}  \left[ \sup_{v\in \Rn,\,F(v)=1}\, \sup_{w\in \Rn\setminus\{0\}} \,\big(\nabla^{2}F^{2}[w](v,v) \big)^{-1}\right]^{-1} \\
 & = \frac{1}{2} \left[ \max_{v\in \Rn,\,F(v)=1}\, \max_{w\in \mathcal{S}^{n-1}} \,\big(\nabla^{2}F^{2}[w](v,v) \big)^{-1}\right]^{-1} \\
 & =  \frac{1}{\sigma^{2}_{F}} \leq 1
 \end{split}
\]
where in the last equality we have used the zero-homogeneity of $\nabla^{2} F^{2}[w]$, the regularity of $F$, and that fact that (by Euler's Theorem)
\[
\sigma^{2}_{F}:=  \max_{v,w \in \mathcal{S}^{n-1}} \,\frac{2 F^{2}(v)}{\nabla^{2}F^{2}[w](v,v)} =   \max_{v,w \in \mathcal{S}^{n-1}} \,\frac{\nabla^{2}F^{2}[v](v,v)}{\nabla^{2}F^{2}[w](v,v)} \geq 1.
\]
Therefore for any $w \in \Rn \setminus \{0\}$  and $v \in \Rn$ with $F(v)=1$ the function
\[
\gamma(t):= F^{2}(w+tv)-\frac{t^{2}}{\sigma^{2}_{F}}
\]
is convex. Then the following 2-uniform convexity easily follows:
\[
 F^2\Big(\frac{\xi_1+\xi_2}{2}\Big)+\frac{1}{\sigma_F^2}F^2\Big(\frac{\xi_1-\xi_2}{2}\Big)
  \leq
   \frac12F^2(\xi_1)+\frac12F^2(\xi_2)\hspace{1em}\forall~\xi_1,\xi_2\in\Rn.
\]
Raising the previous inequality to $p/2$ and using the convexity of the function $f(t)=t^{p/2}$ (since $p\geq 2$) we easily deduce \eqref{A1}.

Now we observe that, the convexity of $F^p$ ensures that
\begin{equation}\label{A2}
 2F^p\Big(\frac{\xi_1+\xi_2}{2}\Big)
  \geq
   2F^p(\xi_1)+pF^{p-1}(\xi_1)F_{\xi}(\xi_1)\cdot(\xi_2-\xi_1)\hspace{1em}\forall~\xi_1,\xi_2\in\Rn.
\end{equation}
	Finally, combining \eqref{A1} and \eqref{A2}, we infer that for any $\xi_1,\xi_2\in\Rn$ it holds
\[
 \begin{split}
  F^p(\xi_2)
   \geq &
     -F^p(\xi_1)+2F^p\Big(\frac{\xi_1+\xi_2}{2}\Big)+\frac{2}{\sigma_F^p}F^p\Big(\frac{\xi_1-\xi_2}{2}\Big)
      \\
       \geq &
        F^p(\xi_1)+pF^{p-1}(\xi_1)F_{\xi}(\xi_1)\cdot(\xi_2-\xi_1)+\frac{2}{\sigma_F^p}F^p\Big(\frac{\xi_1-\xi_2}{2}\Big),
\end{split}
\]
as claimed in \eqref{convex}.  \end{proof}
\begin{remark}
We explicitly remark that the claim of the previous lemma holds even if we drop the $C^2$ regularity assumption on $F$ and we assume $C^1$ regularity and the $p$-uniform convexity \eqref{A1}.
Moreover the previous lemma can be proved also following the line of the \cite[Appendix]{Ln}, by expanding directly the function $f(t):=F^p(\xi_1+t(\xi_2-\xi_1))$ with the Maclaurin's formula and  estimating the second order terms using the $p$-uniform convexity.
\end{remark}

\begin{theorem}\label{Hardy-Sobolev-anis}
Let $\Om\subset\Rn$ be a domain such that that $ r_\Omega < \infty$ and that \eqref{H1} holds true. Suppose there exists $\delta >0$ and $\omega\Subset\Om$ such that $\Om$ is $(\omega,a-\delta)$-feeble-regular with the constant $a$ defined by \eqref{def-a}. Let $\alpha \in [0,1)$ and  $p\in \big[ 2, \frac{n}{1-\alpha} \big)$. Then there exists a positive constant $C_5=C_5(n,p,\alpha,F)$ such that for any $u\in\test(\Om)$ we have
	\begin{equation}\label{thm:HS-ani}
	C_5 \left\|\frac{u}{d_F^{\alpha}}\right\|_{L^{p^{*,\alpha}}(\Om)}^{p} \leq \intom F^p(\nabla u)\dd x-\Big(\frac{p-1}{p}\Big)^p\intom\frac{|u|^p}{d_F^p}\dd x  .
	\end{equation}
\end{theorem}

\begin{proof} In the spirit of \cite{BFT} we obtain an auxiliary lower bound for the {\it anisotropic Hardy difference} defined by
\[
I_p[u;\Om]:=\intom F^p(\nabla u)\dd x-\Big(\frac{p-1}{p}\Big)^p\intom\frac{|u|^p}{d_F^p}\dd x,\hspace{1em}u\in\test(\Om).
\]
To this end we consider the function $v$ defined by
\be\label{gs:transform}
u=d_F^{1-1/p}v,
\ee
and we apply inequality \eqref{convex} to $F(\nabla u)$ with $\xi_1=(1-1/p)d_F^{-1/p}v\nabla d_F $ and $\xi_2=d_F^{1-1/p}\nabla v $. Using \eqref{Fd} and the zero homogeneity of $F_\xi$, we can deduce
\begin{equation}\label{convex2}
\begin{split}
F^p(\nabla u) \ge &  \Big(\frac{p-1}{p}\Big)^p d_F^{-1} |v|^p+\frac{2^{1-p}}{\sigma^p_F}d_F^{p-1} F^p(\nabla v) \\
& +p\Big(\frac{p-1}{p}\Big)^{p-1} |v|^{p-1} F_{\xi}(\nabla d_F )\cdot\nabla v.
\end{split}
\end{equation}
Integrating \eqref{convex2} and recalling \eqref{gs:transform}, we have
\begin{equation}\label{convex3}
I_p[u;\Om]\geq \frac{2^{1-p}}{\sigma^p_F}\int_\Om d_F^{p-1} F^p(\nabla v) \dx+\Big(\frac{p-1}{p}\Big)^{p-1}
\int_\Om  F_{\xi}(\nabla d_F ) \cdot \nabla (|v|^p) \dx.
\end{equation}
The result follows combining Corollary \ref{cor_02}, \eqref{convex3} and \eqref{gs:transform}. 
\end{proof}
\subsection{Anisotropic Hardy-Morrey inequality}
We start by recalling the definition of Morrey spaces.

\begin{definition}[Morrey spaces]
	Assume $\Omega \subset \R^n$ open, $\lambda\geq 0$ and $1\leq p< \infty$. We say that $f\in L^p(\Omega)$ belongs to $L^{p,\lambda}(\Omega)$ if
	\[
	\sup_{0 < r < D_\Omega, x_0\in \Omega} \frac{1}{r^\lambda} \int_{\Omega \cap B_r(x_0)} |f|^p \dx < + \infty,
	\]
	where
	$D_\Omega$ is the diameter of $\Omega$. On $L^{p,\lambda}(\Omega)$ is defined the norm
	\[
	\|f\|_{L^{p,\lambda}(\Om)}:= \Big(\sup_{0 < r < D_\Omega, x_0\in \Omega} \frac{1}{r^\lambda} \int_{\Omega \cap B_r(x_0)} |f|^p \dx \Big)^{1/p}.
	\]
\end{definition}

We have the following theorem

\begin{theorem}\label{thm-HM}
	Under the same assumptions of Theorem \ref{Hardy-Sobolev-anis}, assume in addition that $\Om$ a uniformly Lipschitz domain with finite measure. Then there exists a positive constant $C_6=C_6(n,p,F)$ such that
	\begin{equation}\label{HM-ani}
	\| F(\nabla u) \|_{L^{1, n-n/p}(\Om)} \leq C_6 \big(I_p[u;\Omega]\big)^{1/p}\hspace{1em}\forall~u\in\test(\Om).
	\end{equation}
\end{theorem}

We start the proof with a lemma that provides a weaker form of \eqref{HM-ani}.

\begin{lemma}\label{lemmadistance}
Under the same assumptions of Corollary \ref{cor_02} there exists a constant $C_7=C_7(n,p,\alpha,F)>0$ such that with $\theta:=(p^{*,\alpha})' (1-\alpha)$, there holds
\begin{equation*}
 \| F(\nabla u) \|_{L^{1, n-n/p}(\Om)}
  \leq
   C_7 \big(I_p[u,\Omega]\big)^{1/p} \Big(1 +\big \| d_F^{-1} \big\|^{1-\alpha}_{L^{\theta,n-\theta}(\Om)}\Big)\hspace{1em}\forall~u\in\test(\Om).
\end{equation*}
\end{lemma}

\begin{proof}
	The proof  follows closely that of \cite[Theorem 6.1]{FPs}. Setting $u=d^{1-1/p}v$, by the properties of $F$ we have for a given $x_0\in \Omega$, writing $B_r:=B_r(x_0)$,
	
	\be\nonumber
	\begin{split}
	\int_{\Omega \cap B_r} F(\nabla u)\dd x & \leq \int_{\Omega \cap B_r}d_F^{1-1/p}|F(\nabla v)|\dd x
	+ \frac{p-1}{p}\int_{\Omega \cap B_r}d_F^{-1/p}|v|\dd x \\
	& =: I_1+I_2.
	\end{split}
	\ee
	Using first H\"{o}lder's inequality and then (\ref{convex3}) coupled with the assumption \eqref{H1}, we can write for a positive constant $\kappa(n,p,F)$
	\[
	\begin{split}
	I_1 & \leq  \Big(\int_{\Omega \cap B_r}d_F^{p-1}|F(\nabla v)|^p\dd x\Big)^{1/p}(\omn r^n)^{1-1/p} \\
	& \leq \kappa(n,p,F)\big(I_p[u;\Om]\big)^{1/p}r^{n-n/p},
	\end{split}
	\]
where $\omn$ denote the volume of the unit ball in $\mathbb{R}^n$.
	We will next estimate $I_2$. To this end, we first rewrite it in terms of the original function $u$ and by Holder's inequality
	\begin{align}\nonumber
	I_2
	& =\frac{p-1}{p}
	\int_{\Omega \cap B_r}d_F^{-\alpha}|u|d_F^{\alpha-1}\dd x
	\\ \nonumber & \leq \frac{p-1}{p}
	\Big(\int_{\Omega \cap B_r}\big(d_F^{-\alpha}|u|\big)^{p^{*,\alpha}}\dd x\Big)^{1/p^{*,\alpha}}\Big(\int_{\Omega \cap B_r}d_F^{-(1-\alpha)(p^{*,\alpha})'}\dx\Big)^{1/(p^{*,\alpha})'}.
	\end{align}
	Multiply and divide by $r^{n-n/p}$, noting $n-n/p=(n-\theta)/(p^{*,\alpha})'$ and also using Theorem \ref{Hardy-Sobolev-anis}, we arrive at
	\[
	I_2
	\leq \kappa(n,p,\alpha,F)
	\big(I_p[u;\Om]\big)^{1/p}r^{n-n/p}\big \| d_F^{-1} \big\|^{1-\alpha}_{L^{\theta,n-\theta}(\Om)}.
	\]
	We get the desired result combining the previous estimates.
 \end{proof}

The next ingredient is a technical proposition that is the anisotropic version of \cite[Proposition 5.3]{FPs}.

\begin{proposition}\label{lemma:calculus}
	Let $\Om\subset\Rn,~n\geq2$, be a locally Lipschitz domain satisfying \eqref{H1}. For any $0<\theta<1$ and any $r>0$, defined
	\[
	M_r(\theta):=\int_{B_r\cap\Om}d_F^{-\theta}\dd x
	~~~~\mbox{ and }~~~~Q_r:=\frac{\Hnmo(B_r\cap\prt\Om)}{\Hnmo(\partial B_r)},
	\]
	it holds
	\[
	M_r(\theta)
	\leq
	\frac{\omn}{1-\theta}\Big(n \sup_{\mathbb{S}^{n-1}}| F_\xi| +1-\theta+nQ_r\Big)r^{n-\theta},
	\]
where $\omn$ denote the volume of the unit ball in $\mathbb{R}^n$.
\end{proposition}

\begin{proof} The proof closely follows that of \cite[Proposition 5.3]{FPs} therefore we outline the main ideas and the differences. Writing $\{d_F<r\}$ for the set $\{x\in\Om:d_F(x)<r\}$ and $\{d_F\geq r\}$ for its complement in $\Om$, we have
	\be\label{estim:calc:lemma}
	M_r(\theta)=\int_{B_r\cap\{d_F\geq r\}}d_F^{-\theta}~\dd x+\int_{B_r\cap\{d_F<r\}}d_F^{-\theta}~\dd x.
	\ee
	We can use the monotonicity of $\Ln$ to get the estimate
	\begin{equation}
	\label{estim:calc:lemma:d>r}
	\int_{B_r\cap\{d_F\geq r\}}d_F^{-\theta}~\dd x \leq
	r^{-\theta}\Ln\big(B_r\cap\{d_F\geq r\}\big) \leq
	\omn r^{n-\theta}.
	\end{equation}
	For the second integral of \eqref{estim:calc:lemma} we note first that since $\theta<1$ and $F_\xi(\nabla d_F)\cdot\nabla d_F=1$ a.e. in $\Om$, the generalized Gauss-Green theorem gives\footnote{Since $d_F\in{\rm BV}(\Om)$, its level sets $\{d_F<r\}$ are of finite perimeter for a.e. $r\in(0,\infty)$ (see \cite[Theorem 1-\S5.5]{EvG})}
	\begin{equation*}
	\begin{split}
	(1-\theta)\int_{B_r\cap\{d_F<r\}}d_F^{-\theta}~\dd x
	= & \int_{B_r\cap\{d_F<r\}}F_\xi(\nabla d_F)\cdot\nabla d_F^{1-\theta}~\dd x \\
	= &  \int_{B_r\cap\{d_F<r\}}d_F^{1-\theta}(-\lapdf)\dd x \\
	& + \int_{\prt(B_r\cap\{d_F<r\})}d_F^{1-\theta}F_\xi(\nabla d_F)\cdot\nu\;\dd\Hnmo(x),
	\end{split}
	\end{equation*}
	where $\nu$ is the unit outer normal to $\prt(\Om\cap B_r)$. Using \eqref{H1} the first term on the right hand side is estimated as follows
	\[
	\begin{split}
	\int_{B_r\cap\{d_F<r\}}d_F^{1-\theta}(-\lapdf)\dd x
	& \leq r^{1-\theta}\int_{B_r\cap\{d_F<r\}}(-\lapdf)\dd x
	\\ & =
	-r^{1-\theta}\int_{\prt(B_r\cap\{d_F<r\})} F_\xi(\nabla d_F)\cdot\nu\;\dd\Hnmo(x).
	\end{split}
	\]
	Combining the last two inequalities we arrive at
	\[
	(1-\theta)\int_{B_r\cap\{d_F<r\}}d_F^{-\theta}~\dd x
	\leq
	\int_{\prt(B_r\cap\{d_F<r\})}\big(d_F^{1-\theta}-r^{1-\theta}\big)F_\xi(\nabla d_F)\cdot\nu\;\dd\Hnmo(x).
	\]
	Hence
	\begin{equation*}
	\begin{split}
	(1-\theta)\int_{B_r\cap\{d_F<r\}}d_F^{-\theta}\dd x
	&
	\leq \int_{\prt B_r\cap\{d_F<r\}}\big(r^{1-\theta}-d_F^{1-\theta}\big)| F_\xi(\nabla d_F)\cdot\nu|\dd\Hnmo(x)
	\\  & ~~~~~~
	+ r^{1-\theta}\int_{B_r\cap\prt\Om}| F_\xi(\nabla d_F)\cdot\nu|\dd\Hnmo(x)
	\\  & \hspace{-4em}
	\leq r^{1-\theta}\Big(\int_{\prt B_r\cap\{d_F<r\}}+\int_{B_r\cap\prt\Om}\Big)| F_\xi(\nabla d_F)\cdot\nu|\dd\Hnmo(x)
	\\  & \hspace{-4em} \leq r^{1-\theta}\sup_{\prt B_r\cap\Om}| F_\xi(\nabla d_F)|\Hnmo\big(\prt B_r\cap\{d_F<r\}\big)
	\\  & ~~~~~~ + r^{1-\theta}\Hnmo(B_r\cap \prt\Om)
	\\  & \hspace{-4em} \leq n\omn\sup_{\prt B_r\cap\Om}| F_\xi(\nabla d_F)| r^{n-\theta}
	+ r^{1-\theta}\Hnmo(B_r\cap \prt\Om)
    \\ & \hspace{-4em} \leq n\omn \sup_{\eta\in  \prt B_1(0)}| F_\xi(\eta)|\, r^{n-\theta}
	+ r^{1-\theta}\Hnmo(B_r\cap \prt\Om) ,
	\end{split}
	\end{equation*}
	where in the last estimate we have used the monotonicity of $\Hnmo$ and the zero-homogeneity of $F_\xi$. We conclude by coupling the last inequality with (\ref{estim:calc:lemma}) and   (\ref{estim:calc:lemma:d>r}).
 \end{proof}

\noindent \emph{Proof of Theorem \ref{thm-HM}}
By the assumptions, $\Omega$ satisfies a uniform cone condition that with the finite measure hypothesis implies $\Omega$ is bounded. Therefore, arguing similarly as in  \cite[Lemmas 5.1]{FPs}, $Q_r$ can be estimated uniformly in $r$ by a constant for any $\theta\in(0,1)$. This together with Proposition \ref{lemma:calculus} gives
	\[
	\left \| d_F^{-1} \right\|^{1-\alpha}_{L^{\theta,n-\theta}(\Om)} \leq \kappa(n,p,\alpha,F).
	\]
The result follows by Lemma \ref{lemmadistance}.
\qed

\begin{remark}
	We explicitly remark that in  \cite[Lemma 5.2]{FPs} it is shown that $Q_r$ can be estimated uniformly in $r$ by a constant. Therefore the result of Theorem \ref{thm-HM} continues to hold if we replace the assumptions on  $\Om$ to be a uniformly Lipschitz domain with finite measure with the convexity assumption.
\end{remark}

From classical embedding results on Sobolev, Morrey and Campanato spaces, or arguing more directly using the Morrey's \emph{Dirichlet growth} lemma (as in \cite{FPs}), the following Corollary of Theorem \ref{thm-HM} can be easily proven.

\begin{corollary}
Under the same assumptions of Theorem \ref{thm-HM} we have that for $p>n$, there exists a positive constant $C_8=C_8(n,p,\Om,F)$, such that
\[
 \sup_{\substack{x,y\in\Om \\ x\neq y}}\frac{|u(x)-u(y)|}{|x-y|^{1-n/p}}
  \leq
   C_8\big(I_p[u;\Om]\big)^{1/p}	\hspace{1em}\forall~u\in\test(\Om),
\]
and there exists a positive constant $C_9=C_9(n,p,\Om,F)$ such that
	\[
	\|u\|_{BMO}
	\leq C_9
	\big(I_n[u;\Om]\big)^{1/n}
	\hspace{1em}\forall~u\in\test(\Om).
	\]
\end{corollary}
\noindent{\textbf{Acknowledgements.} G. di Blasio and G. Pisante are members of the Gruppo Nazionale per l'Analisi Matematica, la Probabilit\`{a} e le loro Applicazioni (GNAMPA) of the Istituto Nazionale di Alta Matematica (INdAM) whose support is gratefully acknowledged. G. Psaradakis was supported in part from Universit\`a degli Studi della Campania ``Luigi Vanvitelli" through a visiting researcher position. The authors are grateful to both referees whose remarks and suggestions helped us to considerably enhance the initial version of the article.



\begin{thebibliography}{III}
\bibitem{AFMTV}  Alvino, A., Ferone, A., Mercaldo, A., Takahashi, F., Volpicelli, R.:
 Finsler Hardy-Kato's inequality.
 J. Math. Anal. Appl. 470, 360-374 (2019)

\bibitem{AlFTrL}  Alvino, A., Ferone, V., Trombetti, G., Lions, P.-L.:
 Convex symmetrization and applications.
 Ann. Inst. H. Poincar\'{e} Anal. Non Lin\'{e}aire 14, 275-293 (1997)

\bibitem{Ba} Bal, K.:
 Hardy inequalities for Finsler $p$-Laplacian in the exterior domain.
 Mediterr. J. Math. 14, Art. 165 - 12 pp (2017)

\bibitem{BCL} Ball, K., Carlen, E., Lieb, E.:
 Sharp uniform convexity and smoothness inequalities for trace norms.
 Invent. math.  115, 463-482 (1994)

\bibitem{BFT} Barbatis, G., Filippas, S., Tertikas, A.:
 A unified approach to improved $L^p$ Hardy inequalities with best constants.
 Trans. Amer. Math. Soc. 356, 2169-2196 (2003)

\bibitem{BF} Brasco, L., Franzina, G.:
 Convexity properties of Dirichlet integrals and Picone-type inequalities
 Kodai Math. J. 37, 769-799 (2014)

\bibitem{DdBG} Della Pietra, F., di Blasio, G., Gavitone, N.:
 Anisotropic Hardy inequalities.
 Proc. Roy. Soc. Edinburgh Sect. A 148, 483-498 (2018)

\bibitem{CR-O} Cabr\'{e}, X., Ros-Oton, X.:
 Sobolev and isoperimetric inequalities with monomial weights.
 J. Differential Equations 255, 4312-4336 (2013)

\bibitem{CS} Cianchi, A., Salani, P.:
 Overdetermined anisotropic elliptic problems.
 Math. Ann. 345, 859-881 (2009)

\bibitem{CrM} Crasta, G., Malusa, A.:
 The distance function from the boundary in a Minkowski space.
 Trans. Amer. Math. Soc. 359, 5725-5759 (2007)

\bibitem{DLV}Dyda, B., Lehrb\"{a}ck, J., V\"{a}h\"{a}kangas, A. V.:
 Fractional Hardy-Sobolev type inequalities for half spaces and John domains.
 Proc. Amer. Math. Soc. 146, 3393-3402 (2018)

\bibitem{EvG} Evans, L. C., Gariepy, R. F.:
 Measure Theory and Fine Properties of Funcions.
 Stud. Adv. Math. CRC Press (1991)

\bibitem{FMT} Filippas, S., Maz'ya, V. G., Tertikas, A.:
 Critical Hardy-Sobolev inequalities.
 J. Math. Pures Appl. 87, 37-56 (2007)

\bibitem{FPs} Filippas, S., Psaradakis, G.:
 The Hardy-Morrey \& Hardy-John-Nirenberg inequalities involving distance to the boundary.
 J. Differential Equations 261, 3107-3136 (2016)

\bibitem{FrL} Frank, R. L., Loss, M.:
 Hardy-Sobolev-Maz'ya inequalities for arbitrary domains.
 J. Math. Pures Appl. 97, 39-54 (2011)

\bibitem{GP} Giga, Y., Pisante, G.:
 On representation of boundary integrals involving the mean curvature for mean-convex domains.
 In: Geometric Partial Differential Equations. CRM Series 15, 171-187, Ed. Norm. 2013

\bibitem{Gr} Gromov, M.:
 Sign and geometric meaning of curvature.
 Rend. Semin. Mat. Fis. Milano 61, 9-123 (1991)

\bibitem{KJF} Kufner, A., John, O., Fu\v{c}\'ik, S.:
 Function Spaces. Monographs and Textbooks on Mechanics of Solids and Fluids; Mechanics: Analysis. Noordhoff International Publishing, Leyden; Academia, Prague (1977)

\bibitem{LLL} Lewis, R. T., Li, J., Li, Y.-Y.:
 A geometric characterization of a sharp Hardy inequality.
 J. Funct. Anal. 262, 3159-3185 (2012)

\bibitem{Ln} Lindqvist, P.:
 On the equation $\Div(|\nabla u|^{p-2}\nabla u)+\lambda|u|^{p-2}u=0$.
 Proc. Amer. Math. Soc. 109, 157-164 (1990)

\bibitem{Mz} Maz'ya, V. G.:
 Sobolev spaces. Translated from Russian by Tatyana Shaposhnikova.
 Springer Ser. Soviet Math., Springer (1985)

\bibitem{MSh} Maz'ya, V. G., Shaposhnikova, T.:
 A collection of sharp dilation invariant integral inequalities for differentiable functions.
 In: Sobolev Spaces in Mathematics I, 223-247, Int. Math. Ser. (N. Y.) 8, Springer (2009)

\bibitem{MST} Mercaldo, A., Sano, M., Takahashi, F.:
 Finsler Hardy inequalities.
 arXiv:1806.04901v2

\bibitem{O} Ohta, S.-i.:
 Uniform convexity and smoothness, and their applications in Finsler geometry.
 Math. Ann. 343, 669-699 (2009)

 \bibitem{Ng} Nguyen, V. H.:
 Sharp weighted Sobolev and Gagliardo-Nirenberg inequalities on half spaces via mass transport and consequences.
 Proc. London Math. Soc. 111, 127-138 (2015)

\bibitem{Roc70} Rockafellar, R. T.:
 Convex Analysis.
 Princeton Math. Ser. 28, Princeton (1970)

\bibitem{Ps} Psaradakis, G.:
 $L^1$ Hardy inequalities with weights.
 J. Geom. Anal. 23, 1703-1728 (2013)

\bibitem{Sch13} Schneider, R.:
 Convex Bodies: the Brunn-Minkowski Theory (2nd expanded edition).
 Encyclopedia Math. Appl. 151. Cambridge Univ. Press (2014)

\bibitem{Vn} Van Schaftingen, J.:
 Anisotropic symmetrization.
 Ann. Inst. H. Poincar\'{e} Anal. Non Lin\'{e}aire 23, 539-565 (2006)
\end{thebibliography}
\end{document}